\documentclass[a4paper,11pt]{amsart}
\usepackage{graphicx}
\usepackage{amssymb}
\usepackage{enumitem}

\allowdisplaybreaks[4]

\newtheorem{theorem}{Theorem}
\newtheorem{proposition}{Proposition}
\newtheorem{definition}{Definition}
\theoremstyle{remark}
\newtheorem{remark}{Remark}

\newcommand{\tr}{\operatorname{tr}}

\begin{document}

\title{Introduction to $PSL_2$ phase tropicalization}
\author{Mikhail Shkolnikov$^*$, Peter Petrov \vspace{-16pt}}
\thanks{$^*$Supported by the Simons Foundation International grant no. 992227, IMI-BAS\\Keywords: polar decomposition, hyperbolic amoeba, non-Archimedean tropicalization, circle bundle, phase tropicalization. MSC classes: 14T10, 14T20, 14L35, 14T90, 14H10}
\maketitle

\begin{abstract}
The usual approach to tropical geometry is via degeneration of amoebas of algebraic subvarieties of an algebraic torus $(\mathbb{C}^*)^n$.
An amoeba is logarithmic projection of the variety forgetting the angular part of coordinates, called the phase.
Similar degeneration can be performed without
ignoring the phase. The limit then is called phase tropical variety, and it is a powerful tool in numerous areas. In the
article is described a non-commutative version of phase tropicalization in
the simplest case of the matrix group $PSL_2(\mathbb{C})$, replacing here $(\mathbb{C}^*)^n$ in the classical approach.
\end{abstract}

\section{Introduction}

Tropicalization is defined via amoebas of algebraic varieties. The amoeba of a variety $V \subset (\mathbb{C}^*)^n$ was introduced by
Gelfand, Kapranov, and Zelevinsky in \cite{GKZ}. The map $\operatorname{Log}: (\mathbb{C}^*)^n \rightarrow \mathbb{R}^n, z \mapsto (\log|z_1|,...,\log|z_n|)$, defines the amoeba of $V$ as $\operatorname{Log}(V) \subset \mathbb{R}^n,$ forgetting the angular part called the phase.

G.Mikhalkin and the first author started an investigation of possible generalization of tropicalization, in which the algebraic torus is replaced by a non-commutative complex algebraic group (\cite{MS}). The analog of the map $\operatorname{Log}$ in that case takes values in the
quotient of the group by its maximal compact subgroup, which in the case of $(\mathbb{C}^*)^n$ is $(S^1)^n$. The main example is $PSL_2(\mathbb{C})$, the group of two-by-two complex matrices with unit determinant modulo multiplication by $-1$. This group is isomorphic to the group of orientation-preserving isometries of hyperbolic space $\mathbb{H}^3$. The analog of $\operatorname{Log}$ now is the (orbit) map $\varkappa: PSL_2(\mathbb{C}) \rightarrow \mathbb{H}^3,
\space \varkappa(A) = A(O),$ for a base point $O\in\mathbb{H}^3$, and $A(O)$ denotes
the isometry action defined by $A$. The map $\varkappa$ is seen as taking the quotient of $PSL_2({\mathbb C)}$ by the stabilizer of $O,$ which is the maximal compact subgroup of rotations around $O$.

The hyperbolic amoeba of a subvariety $V\subset PSL_2(\mathbb{C})$ is then defined to be $\varkappa(V)\subset\mathbb{H}^3.$ In \cite{MS}, tropical limits of hyperbolic amoebas of curves were described in terms of certain spherical floor diagrams. However the case of surfaces remains until now unresolved. We announce the following result, with complete proof included in the next extended version of this article.

\textbf{Claim 1.} The tropical limit of hyperbolic amoebas of a family of surfaces in $PSL_2(\mathbb{C})$ is complement to an open ball centered in $O$.

This result shows in particular that one cannot extract much geometric or topological information from this kind of degeneration. One way to enhance it is to use $PSL_2$ phase tropicalization, which does not forget the phase.

 We give a complete characterization for such tropicalizations of constant families (see Theorem \ref{thm_constant}) and consider two examples of non-constant families. In doing this, we use a new type of phase valuation map $\operatorname{VAL},$ arising as a phase tropicalization of a point (see Theorem \ref{thm_VALlim}). Preparations for its explicit description (see Theorem \ref{thm_VALformula}) is the key part of the article. Finally some directions for future development are mentioned, that will become part of the next article, now under preparation.

\section{Introducing $PSL_2(\mathbb{C})$ phase tropicalization}



We choose the following model for the hyperbolic 3-space $\mathbb{H}^3$ (see \cite{Th}). Let $A$ be a two-by-two matrix with complex entries. We denote by $A^*$ its transposed complex conjugate (Hermitian conjugate). The set of all
Hermitian two-by-two matrices ($A = A^*$) is a vector space of real dimension $4$ on which the determinant function defines a quadratic form of signature $(1,3)$. The locus of matrices with determinant $1$ is then a hyperboloid of two sheets, corresponding to positive- and negative-definite matrices. Take $\mathbb{H}^3$ to be the sheet of positive-definite matrices. Note that minus the determinant defines a Riemannian metric of constant negative curvature on it. There is transitive action of $PSL_2(\mathbb{C})$ via isometries on $\mathbb{H}^3$, defined by $A(P) = APA^* \in \mathbb{H}^3$.  Let $O = \begin{pmatrix}
1 & 0 \\
0 & 1
\end{pmatrix}$. The distance from a point $P\in\mathbb{H}^3$ to $O$ is computed as the absolute value of the logarithm of an eigenvalue of $P.$ In particular, taking a power $P^h$ gives a homothety for $h \in \mathbb{R}_{\geq 0}$

\begin{definition} The \emph{hyperbolic amoeba map} is $\varkappa: PSL_2(\mathbb{C}) \rightarrow \mathbb{H}^3, \varkappa(A) := AOA^* = AA^*$. 
\end{definition}
It is a smooth fibration with fibers diffeomorphic to $\mathbb{R}P^3,$ which are cosets of the group $PSU(2)$ acting on $\mathbb{H}^3$ by rotations around $O.$



 A (right) polar decomposition of square matrix $A$ represents it as $A = PU$ where $P$ is a positive semi-definite Hermitian matrix, and $U$ is a unitary matrix. It always exists and is unique for a non-degenerate $A$, with $P$ being the square root of $AA^*,$
and $U = P^{-1}A.$

We have a one-parameter family $R_h$ of diffeomorphisms of $PSL_2(\mathbb{C}),$ lifting the homotheties in $\mathbb{H}^3.$ This diffeomorphisms are defined by $R_h(PU)=P^hU,\ h \in \mathbb{R}_{\geq 0}$, with $P$ a Hermitian positive-definite matrix and $U\in PSU(2)$ unitary matrix considered up to sign. A natural compactification of $PSL_2(\mathbb{C})$ is $\mathbb{C}P^3=\overline{PSL_2(\mathbb{C})}$. Indeed, the identification $\mathbb{C}^4\cong Mat_{2\times 2}(\mathbb{C})$ together with the quotient $\mathbb{C}^4\backslash\{0\}\to \mathbb{C}P^3$, provides an open embedding of $PSL_2(\mathbb{C})=PGL_2(\mathbb{C})$ in $\mathbb{C}P^3$ as the complement of a smooth quadric $Q$. The latter is the projectivization of the set of (non-zero) $2\times 2$ matrices of zero determinant. Then $R_h$ continuously extends to $Q$ as $id_Q$ because taking powers of rank one two-by-two matrix results in proportional matrix.

Let $\mathbb{K}$ be the field of Hahn series, i.e. ``real-power Puiseux series'' (see \cite{M05}) with complex coefficients in variable $t\rightarrow\infty$, and consider $\overline{PSL_2(\mathbb{K})}=\mathbb{K}P^3$. For a non-zero $2\times 2$ matrix $A$ over $\mathbb{K},$ let $[A]_{\mathbb{K}^*}\in\overline{PSL_2(\mathbb{K})}$ denote  the class of all matrices $\mu A$ for  $\mu\in\mathbb{K}^*.$ Assume that all entries in $A$ are convergent power series in $t$. Then $A(t) \neq 0$ for $t$ big enough because for some $\alpha \in \mathbb{R}$, $t^{-\alpha}A$ converges to a non-zero matrix, so it makes sense to consider $[A(t)]_{\mathbb{C}^*}\in\overline{PSL_2(\mathbb{C})}$.

Under the convergence assumption above, it is not difficult to prove the following theorem whose poof will be included in the forthcoming paper. 

\begin{theorem}\label{thm_VALlim}
The limit
$$\lim_{t\rightarrow\infty}R_{(\log(t))^{-1}}[A(t)]_{\mathbb{C}^*}\in\overline{PSL_2(\mathbb{C})}$$
\label{thm_limitexists} exists. We denote it by $\operatorname{VAL}([A]_{\mathbb{K}^*})$ and call it the matrix valuation.
\end{theorem}

In fact this limit depends only on the top order term in $A=Bt^\alpha+o(t^\alpha), \alpha \in \mathbb{R}_{\geq 0}$ for a non-zero matrix $B$ with complex entries, which means that $\operatorname{VAL}([A]_{\mathbb{K}^*})$ makes sense as well in the case when $A$ has divergent entries. In that way we get valuation which is described explicitly below in this section (Theorem \ref{thm_VALexpl}).

\begin{definition} The closure (in Euclidean topology) of the image of any subvariety in $\overline{PSL_2(\mathbb{K})}$ under $\operatorname{VAL}\colon\overline{PSL_2(\mathbb{K})}\rightarrow\overline{PSL_2(\mathbb{C})}$ is called its \emph{$PSL_2$-phase tropicalization.}
\end{definition}

There is a natural compactification $\overline{\mathbb{H}^3}$ of the hyperbolic space with the boundary at infinity identified with $\mathbb{C}P^1$, by considering the image of rank one Hermitian matrices. The amoeba map $\varkappa$ extends to that compactification by the projection of $Q \cong \mathbb{C}P^1 \times \mathbb{C}P^1$ to the first factor\footnote{The quadric $Q$ is parametrized as $([x_0:x_1],[y_0:y_1])\mapsto [x_0y_0:x_0y_1:x_1y_0:x_1y_1].$}, thinking of it as the boundary at infinity of $\mathbb{H}^3$.\\
Also, there is a second amoeba map $\varkappa^*(A)=A^*A$ corresponding to the left quotient by $PSU(2)$ and the left polar decomposition, extended at infinity via projection to the second factor of $Q$. Note that the distances  from $\varkappa(A)$ and $\varkappa^*(A)$ to $O$ coincide.
In both left and right polar decompositions, the unitary parts are equal, giving naturally the following

\begin{definition} The spherical coamoeba map is $\varkappa^\circ\colon PSL_2{\mathbb{C}}\rightarrow PSU(2)$. The double amoeba map is $\widehat{\varkappa}(A)=(\varkappa(A),\varkappa^*(A)), A \in \overline{PSL_2(\mathbb{C})}$. 
\end{definition}

The image of  $\widehat{\varkappa}$ is the subvariety $\widehat{Q} \subset \overline{\mathbb{H}^3} \times \overline{\mathbb{H}^3}$ of real dimension 5 of all pairs of points at equal distance to $O$.

\begin{remark}If $\det(A)=1,$ $\varkappa^\circ([A]_{\mathbb{C}^*})=[A+(A^*)^{-1}]_{\mathbb{R}^*}\in\mathbb{R}P^3\cong PSU(2).$
\end{remark}

For a point $P\in\mathbb{H}^3\backslash\{O\}$, denote by $P^\infty$ its projection to the boundary $\mathbb{C}P^1$ by the geodesic ray emanating from $O$ and passing through $P.$ We think of $\mathbb{H}^3$ as a cone over $\mathbb{C}P^1$ with the vertex $O$, getting for all other points $P$ pairs $(d(P,O),P^\infty)$. A point of $(P_1,P_2)\in\mathbb{H}^3\times\mathbb{H}^3$ from $\widehat{Q},$ i.e. satisfying $d(P_1,O)=d(P_2,O),$ may be identified with a point of a cone over $\mathbb{C}P^1\times\mathbb{C}P^1\cong Q,$ with the vertex corresponding to  $P_1=P_2=O$ and all other points corresponding to $(d(P_1,O),(P_1^\infty,P_2^\infty))$ in $(0,\infty)\times Q$.\\
This represents $\widehat{Q}$ as a real cone over $Q$. The fiber of the double amoeba map $\widehat{\varkappa}$ over the vertex $O$ is $PSU(2) \simeq \mathbb{R}P^3$, the fiber over its base $\{\infty\}\times Q$ is a single point, and the fiber over all other points is a circle. We describe now the corresponding circle bundle.\\

Take a point $(\alpha,(B_1,B_2))\in\widehat{Q}$ with $\alpha>0$ and $B_1,B_2\in\mathbb{C}P^1$. The fiber of $\widehat{\varkappa}$ over this point is naturally isomorphic to $\mathcal{S}_{(B_1,B_2)} = \{U\in PSU(2) \ |B_1=UB_2U^{-1}\}$, corresponding to all rotations around $O$ and sending $B_2$ to $B_1.$ In particular, we see that it doesn't depend on $\alpha$, and this gives rise to the circle bundle $\mathcal{S}$ over $Q.$ The following lemma characterizes $\mathcal{S}$. 

\begin{proposition} The fiber  of $\mathcal{S}$ over a point $[B]_{\mathbb{C}^*}\in Q$ is naturally isomorphic to $\{[cB]_{\mathbb{R}^*}|c\in\mathbb{C}^*\}.$\end{proposition}

\begin{proof} A geometric way to see this isomorphism is via the classical identification of the complex quadric surface $Q$ and oriented real lines in projective three-space, which is $PSU(2)$. Such a line represents a coset in this group and corresponds to the fiber over $[B]_{\mathbb{C}^*}$ of $\mathcal{S}$.
\end{proof}


\begin{theorem}
\label{thm_VALformula}
 Consider $[A]_{\mathbb{K}^*}\in\overline{PSL_2(\mathbb{K})}$, with $A=Bt^\alpha+o(t^\alpha), \alpha \in \mathbb{R}_{\geq 0}$. If $\det(A)=0$, then $\operatorname{VAL}([A]_{\mathbb{K}^*})=(\infty, [B]_{\mathbb{C}^*}).$ Otherwise, assume that $A$ is normalized so that $\det(A)=1$. If $\alpha>0$ we have $\det(B)=0$, and then $\operatorname{VAL}([A])=(\alpha, [B]_{\mathbb{R}^*})$. If $\alpha=0$ and then $\det(B)=1$, we have $\operatorname{VAL}(A)=(0, \varkappa^\circ([B]_{\mathbb{C}^*})).$
\label{thm_VALexpl}
\end{theorem}


\begin{remark}In this theorem, we are using the cone representation of $\mathbb{C}P^3$ in which it appears as fibration over $\widehat{Q}$ with the fiber over the vertex being $PSU(2),$ the base identified with $Q,$ and all other points having the form $(\alpha,[B]_{\mathbb{R}^*})$ with $\alpha>0$ and $det (B)=0.$ The latter type of points can be explicitly mapped to $PSL_2\mathbb{C}$ by
$$(\alpha,[B]_{\mathbb{R}^*})\mapsto [e^\alpha B+e^{-\alpha}(B^c)^*)]_{\mathbb{C}^*},$$ where $B^c=\begin{pmatrix}d&-b\\-c&a\end{pmatrix}$ is the adjugate matrix for $B=\begin{pmatrix}a&b\\c&d\end{pmatrix}.$ This demonstrates the advantage of the cone picture where the expression of $\operatorname{VAL}$ has a much simpler form than it has in the standard matrix presentation.
\end{remark}

Next is described the tropicalization of a constant family. For variety $V$ defined over $\mathbb{C}$ take its $\mathbb{K}$-points $V(\mathbb{K}).$

\begin{theorem}
\label{thm_constant}
 Let $V$ be a subvariety of $\overline{PSL_2(\mathbb{C})}$ with no irreducible components in $Q.$ Then the following representation holds:$$\operatorname{VAL}(V(\mathbb K))=\{0\}\times\varkappa^\circ(V\cap PSL_2(\mathbb{C}))\cup(0,\infty)\times\mathcal{S}|_{V\cap Q}\cup \{\infty\}\times (V\cap Q),$$ where $\varkappa^\circ$ is the spherical coamoeba map, and $\mathcal{S}|_{V\cap Q}$ is the total space of the circle bundle $\mathcal{S}$ restricted to $V\cap Q$.\end{theorem}

\begin{proof} Take an $A(t)\in V(\mathbb{K})$ represented as $A(t)=t^\alpha B+o(t^\alpha)$. Then $[B]_{\mathbb{C}^*}\in V$, because for any homogeneous $f$ of degree $d$ vanishing on $V$ we get $0=f(A(t))=t^{d\alpha}f(B)+o(t^{d\alpha}),$ that is, $f(B) = 0$. Similarly, if $\det(A(t))=0$ we have $\det(B)=0$, that is $[B]_{\mathbb{C}^*}\in V\cap Q$, and $\operatorname{VAL}(A(t))=(\infty,[B]_{\mathbb{C}^*})$ belongs to the first component of the union. This proves one inclusion.

For the opposite inclusion, take an $A(t)$ for each element of the right-hand side. An element in the first set is some $(0,\varkappa^\circ(B))$, so put $A(t) = B$. An element of the third set is $(\infty,[B]_{\mathbb{C}^*})$ for some $[B]_{\mathbb{C}^*}\in V\cap Q,$ so take $A(t)=B.$

An element from the middle term is $(\alpha, [B]_{\mathbb{R}^*})$,
for some $\alpha>0$ and $[B]_{\mathbb{C}^*}\in V\cap Q$. Since no component of $V$ is contained in $Q$ there is a curve $C \subset V$ through $[B]_{\mathbb{C}^*}$, not contained in $Q$. Pick up local parametrization for it when $t\rightarrow \infty$ near the point $[B]_{\mathbb{C}^*}$. This parametrization is given by a matrix power series $A(t)$ with $\det(A(t))$ not identically zero power series. Normalize to have $\det(A(t))=1$. Then $A(t)=t^\gamma c B+o(t^\gamma)$ with $\gamma>0,$ and $c\in\mathbb{C}^*$. Make change of variables in $A(t)$ giving $\tilde{A}(t)$ by replacing each $t^\beta$ by $\exp(-\beta\gamma^{-1} \sigma)t^{\alpha\beta\gamma^{-1}},$ where $\sigma=\log(c)$. This comes from the necessity to obtain the image of the term $ct^{\gamma}$ in the desired form as $t^\alpha$, and it is an automorphism of $\mathbb{K}$ fixing $\mathbb{C}$.
Then $\tilde A(t)=t^\alpha B+o(t^\alpha) \in V(\mathbb{K})\cap PSL_2(\mathbb{K})$ and $\operatorname{VAL}(\tilde A(t))=(\alpha,[B]_{\mathbb{R}^*})$. \end{proof}

\section{\vspace{-2pt}Examples}
\begin{figure}
\includegraphics[width=0.8\textwidth]{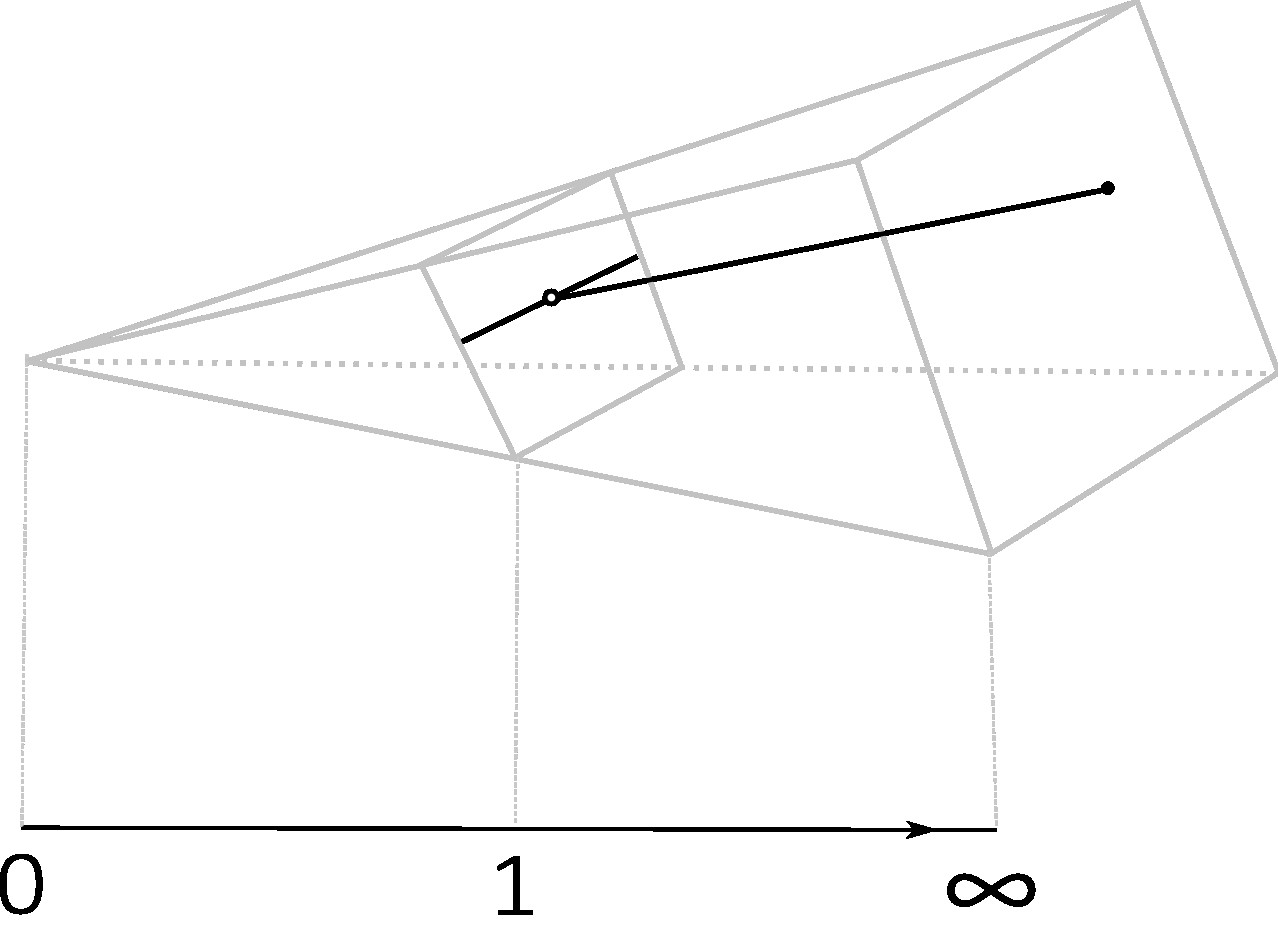}
\caption{This picture represents the image under $\operatorname{VAL}$ of a line tangent to the quadric $Q(\mathbb{K})$.}\label{fig_line}
\end{figure}
\textbf{Example 1.} We describe the image under $\operatorname{VAL}$ of a particular line $L$ tangent to the quadric $Q(\mathbb{K}).$ Consider $L\subset\mathbb{K}P^3$ parametrized by $$\mathbb{K}P^1\ni[z:w]\longmapsto Z(t)=\begin{bmatrix}
tw & z \\
0 & t^{-1}w
\end{bmatrix}\in\overline{\operatorname{PSL_2}\mathbb{K}}=\mathbb{K}P^3,$$ $t\in\mathbb{K}$ being the uniformizer. The point where $L$ intersects the quadric $Q(\mathbb{K})$ is $B_\infty=\begin{bmatrix}
0 & 1 \\
0 & 0
\end{bmatrix}$ corresponding to $w=0.$

For all other values of $[z:1]\in\mathbb{K}P^1$ and $z=ct^\gamma+o(t^\gamma),$ $c\in\mathbb{C}^*,$ we have three cases: $\gamma>1$, $\gamma = 1$ and $\gamma < 1$. Look for the limit at $\infty$ for $z$ represented by decreasing exponents.\\
In the first case, the dominant term of $Z(t)$ is $\begin{bmatrix}
0 & ct^\gamma \\
0 & 0
\end{bmatrix}.$ Projected by $\operatorname{VAL}$ on the cone picture of $\overline{PSL}_2\mathbb{C}=\mathbb{C}P^3$, it is a ray over $B_\infty$, each point with an $S^1$ in the circle bundle $\mathcal{S}$.\\
In the second case, the dominant term is $\begin{bmatrix}
t & ct\\
0 & 0
\end{bmatrix}$, giving a section of the $\mathcal{S}$ bundle over $Q(\mathbb{C})$ with points $\begin{bmatrix}
1 & 0 \\
0 & 0
\end{bmatrix}$ and $\begin{bmatrix}
0 & 1 \\
0 & 0
\end{bmatrix}$ excluded.\\

In the last case, the dominant term of $Z(t)$ is $\begin{bmatrix}
t & 0 \\
0 & 0
\end{bmatrix}$ giving a single point, which is one of the two points excluded in the previous case.\\
We note that the image is not connected (see Figure \ref{fig_line}) because the phase at any point $(1,B_\infty)$ with $\alpha= 1$ is missing.
This is a typical feature of images in $\mathbb{K}P^3$ under $\operatorname{VAL}.$\\

\textbf{Example 2.}
\begin{figure}
\includegraphics[width=0.8\textwidth]{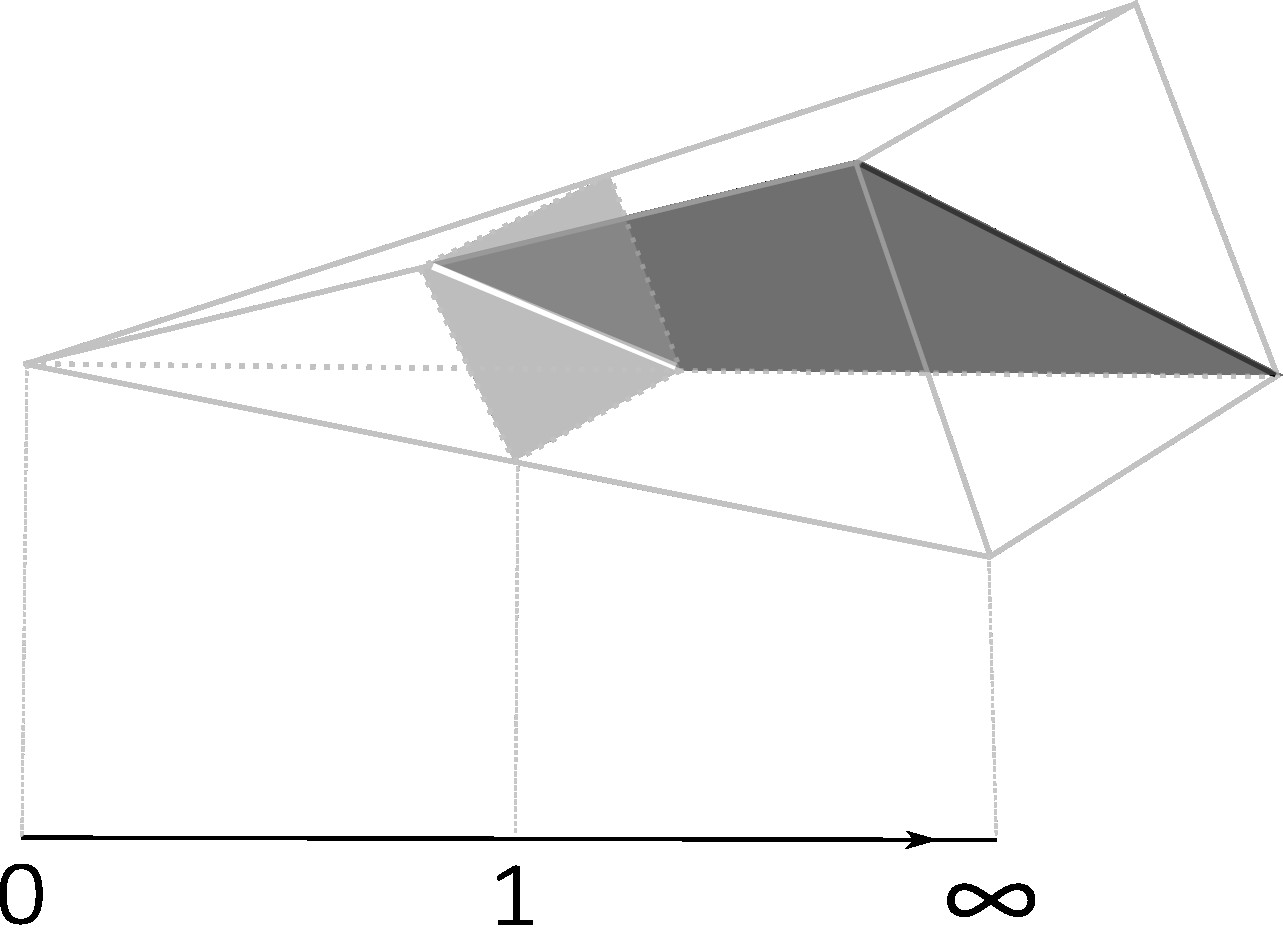}
\caption{The image of a quadric surface under $\operatorname{VAL}.$ Note that a curve at the critical height 1 is missing.}
\label{fig_exquadric}
\end{figure}
Consider the quadric surface $S$ over $\mathbb{K}$ defined by
 \begin{equation}\label{eq_aquadric} t^2\det(A(t))=(\tr(A(t)))^2,\end{equation}
where determinant and trace are considered polynomials in two-by-two matrices.\\

\begin{proposition}
The image $\operatorname{VAL}(S)$ (see Figure \ref{fig_exquadric}) is
\begin{equation}\label{eq_VALaquadric}\{1\}\times s(Q\backslash C)\cup(1,\infty)\times\mathcal{S}|_C\cup \{\infty\}\times C,\end{equation}
where $C$ is a curve on $Q$ given by $\tr=0,$ $s([B]_{\mathbb{C}^*}):=[(\tr(B))^{-1}B]_{\mathbb{R}^*}$ is a section of $\mathcal{S}$ defined away from $C$, and $\mathcal{S}|_C$ is the restriction of $\mathcal{S}$ to $C.$\end{proposition}

\begin{proof} First we show the inclusion of $\operatorname{VAL}(S)$ in (\ref{eq_VALaquadric}). Consider $A(t)=t^\alpha B+o(t^\alpha)$ such that $\det(A(t))=1$ and satisfying (\ref{eq_aquadric}). Then the top-degree part of the left-hand side of the defining equation for $S$ becomes $(\tr(B))^2t^{2\alpha} + o(t^{2\alpha})$. It implies that $\alpha\geq 1$ because otherwise, the right-hand side of (\ref{eq_aquadric}) would be the only dominant term and the equation doesn't hold. The same is true in the case $\alpha=1$ and $\tr(B)=0,$ which explains why  $C$ is missing when $\alpha = 1$. If $\alpha=1$ and $\tr(B)\neq 0$ for the equation (\ref{eq_aquadric}) could exist solution $A(t)=Bt+o(t)$ for which $\tr(B) = \pm 1$, i.e. $\operatorname{VAL}(A(t))$ must be in the image of the section $s([B]_{\mathbb{C}^*})=[(\tr(B))^{-1}B]_{\mathbb{R}^*}$. If $\alpha>1$ we have again a single dominant term, but now coming from the right-hand side, unless $\tr(B)=0$, and it belongs to the second component of the union. Finally, if $\det(A(t))=0$ we get that $\operatorname{VAL}(A(t))$ belongs to the third component of the union because the top-order part of $A(t)$ satisfies $(\tr(B))^2=0$.

For the other inclusion we need to find an $A(t)$ with $\operatorname{VAL}(A(t))=p$ for every $p$ in (\ref{eq_VALaquadric}). As $\operatorname{VAL}$ is equivariant with respect to both left and right $PSU(2)$ action, and the equation of $S$ is symmetric by the conjugation by $PSU(2)$, the problem is reduced to orbits. The orbits of the images by $\operatorname{VAL}$ of the following sets of points on $S$ by this action span the whole (\ref{eq_VALaquadric}):\begin{itemize}[leftmargin=*]
 \item $\alpha=\infty$,  $\begin{bmatrix}0 & 1\\ 0& 0\end{bmatrix}_{\mathbb{K}^*}$;
\item $\alpha\in(1,\infty),$  $\begin{bmatrix}t & \sigma t^\alpha\\ -{(\sigma t^\alpha)}^{-1}& 0\end{bmatrix}_{\mathbb{K}^*}$ for $\sigma\in\mathbb{C}^*$;
\item $\alpha=1$, $\begin{bmatrix}t-t^{-1}&-t^{-1}\\t^{-1}&t^{-1}\end{bmatrix}_{\mathbb{K}^*}$ and $\begin{bmatrix}t-t^{-1}&\sigma^{-1}t^{-3}\\-\sigma t&t^{-1}\end{bmatrix}_{\mathbb{K}^*}$ for $\sigma\in\mathbb{C}^*$.
\end{itemize}

\end{proof}



\section{Further development}
The tropicalization for groups different from $PSL_2(\mathbb{C})$ is a meaningful direction for subsequent research. One principle we follow is that our tropicalization procedure is not about groups or matrices, but is rather dictated by order of degenerations of subvarieties on a quadric, given here by the equation $\det=0$. In particular, it makes sense in any dimension of the projective space with a spherical coamoeba map only, but not amoeba map and the non-phase tropicalization, surviving in this generalization. The first development we plan to make is to explore the case of the complex projective plane with a conic. For instance,  the relation of our degeneration diagrams with a version of floor diagrams introduced by Brugallé (\cite{B15}) will be clarified. It is expected that our techniques provide a convenient tool to investigate topological invariants of complex varieties in the spirit of (\cite{KZ}) and their complements, and to count geometric objects like curves, first proving the necessary correspondence theorems.\vspace{8pt}



$\mbox{\textbf{Acknowledgements}}$ The authors would like to express their sincere gratitude to Velichka Milousheva for enduring support and encuragement.\vspace{-4pt}

\bibliographystyle{amsplain}

\ \\
\noindent\ \\
Mikhail Shkolnikov, Institute of Mathematics and Informatics, Bulgarian Academy of Sciences; Akad. G. Bonchev St, Bl. 8, 1113 Sofia, Bulgaria; \\email: m.shkolnikov@math.bas.bg\vspace{5pt}\\ 
Peter Petrov, Institute of Mathematics and Informatics, Bulgarian Academy of Sciences;
Akad. G. Bonchev St, Bl. 8, 1113 Sofia, Bulgaria;
\\email: pk5rov@gmail.com

\end{document}